\documentclass{article}

\usepackage{amsrefs}
\usepackage{amsmath,amssymb,amsthm}
\usepackage{enumerate,pxfonts,tikz}
\usepackage[all,cmtip]{xy}
\usepackage[applemac]{inputenc}
\usepackage[colorlinks=true, linkcolor=black, citecolor=black]{hyperref}

\usepackage[OT1]{fontenc}
\DeclareFontFamily{OT1}{pzc}{}
\DeclareFontShape{OT1}{pzc}{m}{it}{<-> s * [1.35] pzcmi7t}{}
\DeclareMathAlphabet{\mathcal}{OT1}{pzc}{m}{it}

\newcommand \al{\alpha}
\newcommand \be{\beta}
\newcommand \bs{\backslash}

\newcommand \CA{\mathcal{A}}

\newcommand \CC{\mathcal{C}}

\newcommand \CD{\mathcal{D}}

\newcommand \cent{\operatorname{cent}}

\newcommand \eps{\varepsilon}
\newcommand \F{\mathbb{F}}
\newcommand \ga{\gamma}
\newcommand \Ga{\Gamma}
\newcommand \GL{\operatorname{GL}}

\newcommand \la{\lambda}
\newcommand \La{\Lambda}
\newcommand \M{\operatorname{M}}

\newcommand \PSL{\operatorname{PSL}}

\newcommand \R{{\mathbb R}}
\newcommand \SL{\operatorname{SL}}

\newcommand \tr{\operatorname{tr}}

\newcommand \Z{{\mathbb Z}}

\renewcommand \H{\mathbb{H}}
\renewcommand \1{{\bf 1}}

\renewcommand \({\big(}
\renewcommand \){\big)}
\renewcommand \[{\left(}
\renewcommand \]{\right)}

\newcommand{\e}
[1]{\emph{#1}\index{#1}}

\newcommand{\smat}
[4]{\(\begin{smallmatrix}#1 & #2 \\ #3 & #4\end{smallmatrix}\)}

\renewcommand{\sp}
[1]{\left\langle #1\right\rangle}

\newtheorem{theorem}{Theorem}[section]

\newtheorem{lemma}[theorem]{Lemma}
\newtheorem{corollary}[theorem]{Corollary}

\theoremstyle{definition}
\newtheorem{definition}[theorem]{Definition}

\begin{document}

\pagestyle{myheadings} \markright{Distribution of traces}

\title{Distribution of  prime geodesic traces}
\author{Anton Deitmar}
\date{}
\maketitle

{\bf Abstract:}
This note complements a recent paper of Chatzakos, Harcos and Kaneko \cite{CHK}.
We use a Dirichlet style Prime Geodesic Theorem to improve the error term estimate in loc. cit. at the cost of lowering the resolution. 
This outcome is obtained by using available results from literature and a little calculation.

$$ $$

\tableofcontents


\section*{Introduction}

For $x>0$ let 
$$
\psi(x)=\sum_{u\le x}h(u)\log(u_0),
$$
where the sum runs over all norm one units $u>1$ of real quadratic orders, $h(u)$ is the class number and $u_0>1$ is the fundamental unit in the narrower sense.
It is known \cite{SY} that 
$
\psi(x)=x+O_\eps\[x^{25/36+\eps}\]
$
holds for $x\to\infty$.
For given $u$, the trace $\tr(u)=u+u^c=u+\frac1u$ is a positive integer and 
we are concerned with the distribution of these numbers modulo a given prime number $p\ge 3$.
For  given $a\in\F_p$, define
$$
\psi_{a}(x)= \sum_{\substack{1<u\le x\\ \tr(u)\equiv a\ (p)}}h(u)\log(u_0).
$$
In the light of Dirichlet's Prime Number Theorem, one might expect an equidistribution of the residue classes, however, in \cite{CHK} it is shown that
$$
\psi_{a}(x)=\begin{cases}\frac x{p-1}+O_\varepsilon\big(x^{\frac34+\frac1{12}+\varepsilon}\big)
&\text{if }\left(\frac{a^2-4}p\right)=1,\\
\frac x{p+1}+O_\varepsilon\big(x^{\frac34+\frac1{12}+\varepsilon}\big)
&\text{if }\left(\frac{a^2-4}p\right)=-1,\\
\frac {px}{p^2-1}+O_\varepsilon\big(x^{\frac34+\frac1{12}+\varepsilon}\big)
&\text{if }\left(\frac{a^2-4}p\right)=0.\\
\end{cases}
$$
To prove this, the authors  consider binary quadratic forms and their connection to Dirichlet L-functions given by the class number formula.

In this paper, we use results of Selberg and Hejhal to get a better error term estimate at the cost of having to combine the contributions of $a$ and $-a$.
Choosing an integral basis of a quadratic number field, i.e., an embedding into $\GL(2)$, the unit $u$ gives an element of $\SL_2(\Z)$.
An asymtotic result by Hejhal \cite{Hej2} directly yields an equidistribution law, which one might call \e{Dirichlet Prime Geodesic Theorem} and which, in the cocompact case, also is contained in Peter Sarnak's PhD thesis.
Based on the Dirichlet Theorem, the distribution of  traces is easily determined by an exercise in linear algebra, determining the conjugacy classes in the group $\SL_2(\F_p)$.
In this way, we bring the above error term down to $O\(x^{\frac34}\log x\)$, at least for $\psi_a+\psi_{-a}$ in place of $\psi_a$.

\section{Statement of the main assertion}

In the introduction, we have used the dictionary, which relates units in quadratic orders to hyperbolic elements of $\SL_2(\Z)$ or $\PSL_2(\Z)$.
This dictionary is obtained through the choice of an integral basis of a number field.
See \cite{DH} for further explanations.
For the rest of the paper we shall express these items in terms of the arithmetic groups  $\SL_2(\Z)$ and $\PSL_2(\Z)$.

\begin{definition}
For a hyperbolic conjugacy class $[\ga]$ in $\SL_2(\Z)$ we let $\ell(\ga)$ denote the length of the geodesic in the upper half plane $\H$, which is closed by $\ga$.
Further, let $\ga_0$ denote the underlying primitive hyperbolic element.
\end{definition}

\begin{definition}
For a prime number $p$,  an element $a\in\Z/p\Z$ and $x>0$ let
\begin{align*}
\psi^\pm_a(x)=\frac12\sum_{\substack{[\ga]\subset\SL_2(\Z)\\ \ell(\ga)\le\log x\\ \tr(\ga)\equiv a\ (p)}}\ell(\ga_0),
\end{align*}
where the sum runs over  hyperbolic conjugacy classes $[\ga]$ in $\SL_2(\Z)$.
\end{definition}

\begin{lemma}
We have 
$$
\psi^\pm_a=\frac{\psi_a+\psi_{-a}}2.
$$
\end{lemma}

\begin{proof}
In the sum defining $\psi_a$ we only considered $u>1$, which corresponds to the conjugacy classes $[\ga]$ with $\tr(\ga)>0$. As all classes occur in the sum $\psi_a^\pm$, the lemma follows.
\end{proof}

\begin{theorem}\label{thmMain}
If $p\ge 3$, one has, as $x\to\infty$, \label{thmMain}
$$
\psi^\pm_{a}(x)=\begin{cases}\frac x{p-1}+O\(x^{\frac34}(\log x)^{\frac12}\)
&\text{if }\left(\frac{a^2-4}p\right)=1,\\
\frac x{p+1}+O\(x^{\frac34}(\log x)^{\frac12}\)
&\text{if }\left(\frac{a^2-4}p\right)=-1,\\
\frac {px}{p^2-1}+O\(x^{\frac34}(\log x)^{\frac12}\)
&\text{if }\left(\frac{a^2-4}p\right)=0.\\
\end{cases}
$$
For the prime $p=2$ we have 
$$
\psi^\pm_{a}(x)=\psi_{a}(x)=\begin{cases}
\frac13x+O\(x^{\frac34}(\log x)^{\frac12}\)
&\text{if }a=1,\\
\frac 23x+O\big(x^{\frac34}(\log x)^{\frac12}\big)
&\text{if }a=0.\\
\end{cases}
$$
\end{theorem} 

The proof will be given in the following sections.

\section{The Dirichlet Theorem for PSL(2)}

Let $\La$ denote a lattice in $\PSL_2(\R)$ and let $\La'\subset \La$ be a normal sublattice.

\begin{definition}
The finite group $\La/\La'$ acts unitarily on the Hilbert space $L^2(\La'\bs\H)$.
Let $\chi$ be an irreducible representation of  $\La/\La'$ and let $L^2(\La\bs \H)(\chi)$ denote the $\chi$-isotype.
We shall also consider $\chi$ as a representation of $\La$ whose kernel contains $\La'$.
As the hyperbolic Laplacian $\Delta$ is invariant, it induces a densely defined operator $\Delta_\chi$ on $L^2(\La\bs\H)(\chi)$.

Let $\la_{\chi,1}\le\dots\le\la_{\chi,n(\chi)}$ be the Laplace eigenvalues in $(0,3/16)$ on $L^2(\La\bs\H)(\chi)$.
Let $s_j=s_{\chi,j}>\frac34$ be defined by  $\la_j=s_j(1-s_j)$.
\end{definition}

\begin{lemma}\label{lem2.4}
Let $\La$ be a congruence group in $\PSL_2(\Z)$ and suppose that $\ker(\chi)$ is a congruence group, then the interval $\[\frac7{10},1\]$ does not contain any $s_{\chi,j}$.
\end{lemma}

\begin{proof}
The assertion of the lemma is equivalent to saying that the Laplace operator $\Delta_\chi$ on   $L^2(\La\bs \H)(\chi)$ has no eigenvalue in $\[0,\frac{21}{100}\]$. We say, it has no critical eigenvalue.
Let $\La'$ be the kernel of $\chi$. In \cite{LRS} it is shown that the Laplacian on $L^2(\La'\bs\H)$ has no critical eigenvalue.
The finite group $\La/\La'$ acts on $L^2(\La'\bs \H)$ and so the latter decomposes into isotypes
$$
L^2(\La'\bs \H)=\bigoplus_{\chi\in\widehat{\La/\La'}}L^2(\La\bs\H)(\chi).
$$
Since the Laplacian has no critical eigenvalue in $L^2(\La'\bs\H)$, it has no critical eigenvalue on each $L^2(\La\bs\H,\chi)$.
\end{proof}

\begin{definition}
Let   $\La\subset G$ be a lattice, $\La'\subset \La$ be a normal subgroup of finite index and $\CC$ is a conjugacy class in $\La/\La'$.
Define
$$
\psi_{\La,\CC}(x)=\sum_{\substack{\ell(\ga)\le\log x\\ [\ga]\equiv \CC\mod\La'}}\ell(\ga_0).
$$
\end{definition}

\begin{theorem}[Dirichlet Theorem]\label{thmDiri}
Let $\La',\La$ be lattices in  $\PSL_2(\R)$, where $\La'$ is normal in $\La$. Let $\CC$ denote a conjugacy class in the finite group $\La/\La'$.
Then
$$
\psi_{\La,\CC}(x)=\frac{|\CC|}{|\La/\La'|}\[x+ \sum_{\chi}\tr\chi(\CC)\sum_{\substack{j=1,\dots,n(\chi)\\ s_{\chi,j}<1}}\frac{x^{s_{\chi,j}}}{s_{\chi,j}}\]+O\(x^{\frac34}(\log x)^{\frac12}\).
$$
where the sum runs over all irreducible representations $\chi$ of the finite group $\La/\La'$.
If $\La$ and $\La'$ are congruence subgroups of $\PSL_2(\Z)$, then
$$
\psi_{\La,\CC}(x)=\frac{|\CC|}{|\La/\La'|}x+O\(x^{\frac34}(\log x)^{\frac12}\).
$$
\end{theorem}

\begin{proof}
On the finite group $\La/\La'$ we install the counting measure divided by $|\La/\La'|$.
Then the functions $\tr\chi$ with $\chi\in\widehat{\La/\La'}$ form an orthonormal basis of $L^2\((\La/\La')/\text{conjugation}\)$. For a given conjugacy class $\CC$ we write $\tr\chi(\CC)$ for the trace $\tr\chi(x)$ for any member $x$ of $\CC$.
We get
\begin{align*}
\1_\CC(y)&=\sum_\chi\tr\chi(y)\ \sp{\1_\CC,\chi}=\sum_{\chi}\tr\chi(y)\frac{|\CC|}{|\La/\La'|}\tr\chi(\CC).
\end{align*}
This implies
$
\psi_{\La,\CC}=\frac{|\CC|}{|\La/\La'|}\sum_\chi\tr\chi(y)\ \psi_{\La,\chi}.
$
Write 
$
\psi_{\La,\chi}(x)=\sum_{\ell(\ga)\le\log x}\ell(\ga_0)\ \tr\chi(\ga),
$
where the sum runs over all hyperbolic conjugacy classes $[\ga]$ in $\La$.
In \cite{Hej2}, Theorem 3.4 on page 474 we find
$
\psi_{\La,\chi}(x)=\sum_{j=1}^{n(\chi)}\frac{x^{s_j}}{s_j}+O\(x^{\frac34}(\log x)^\frac12\).
$
The eigenvalue $\la=0$, which corresponds to $s=1$ only occurs for $\chi=1$, where it has multiplicity 1.
Therefore we can split off the term of this eigenvalue and the claim follows.
In the case of congruence subgroups the additional claim is implied by Lemma  \ref{lem2.4}.
\end{proof}

\section{Lift to SL(2)}

We need to clarify the relation between conjugacy classes in $\SL(2)$ and $\PSL(2)$.
We first stick to the general situation of arbitrary lattices and  specialise to the arithmetic case in the next section.

\begin{definition}
Let $\Ga\subset\SL_2(\R)$ denote the preimage of the lattice $\La\subset\PSL_2(\R)$.
Assume that there exists a normal subgroup $\Ga'\subset\Ga$ such that the projection $\SL_2\to\PSL_2$ maps $\Ga'$ isomorphically to $\La'$.
It then follows that $-1\notin\Ga'$ and the projection $\pi:\Ga/\Ga'\to \La/\La'$  has kernel $\{\pm 1\}$.
For each conjugacy class $\CD$ of $\Ga/\Ga'$, let $n(\CD)$ denote the number of classes in $\Ga/\Ga'$ which map to the conjugacy class $\pi(\CD)\subset \Ga/\Ga'$.
This means that $n(\CD)=1$ if $\CD=-\CD$ and it is $2$ otherwise.
We define
$$
\psi_{\Ga,\CD}(x)=\sum_{\substack{[\la]\equiv \CD\mod\Ga'\\ \ell(\la)\le\log x}}\ell(\la_0),
$$
where the sum now runs over hyperbolic conjugacy classes $[\la]$ in $\Ga$ and $\la_0$ is a generator (modulo $\pm 1$) of the centraliser $\Ga_\la$ of $\la$ in $\Ga$.
Over every hyperbolic class $[\ga]$ in $\La$ there lie two conjugacy classes $[\la]$ in $\Ga$, since $\tr(\la)\ne 0$. 
We infer
$$
\psi_{\Ga,\CD}=\psi_{\Ga,-\CD}=\(3-n(\CD)\)\psi_{\La,\pi(\CD)}.
$$
The cardinalities of the respective conjugacy classes are related by 
$$
|\CD|=\left\{\begin{array}{cl} |\pi(\CD)|&\text{if }n(\CD)=2,\\ 2|\pi(\CD)|&\text{if }n(\CD)=1,\end{array}\right\}=\(3-n(\CD)\)\ |\pi(\CD)|.
$$
Write $V(\CD)=\(3-n(\CD)\)\frac{|\pi(\CD)|}{|\La/\La'|}$. Then we get
\begin{align*}
V(\CD)=\frac{|\CD|}{|\Ga/\Ga'|/2}=\frac{2}{|\cent(A)|},\tag*{$(*)$}
\end{align*}
where $\cent(A)$ is the centraliser of $A\in \CD$ in the group $\Ga/\Ga'$.
\end{definition}

\begin{corollary}\label{cor3.3}
Let $\CD$ denote a conjugacy class in the finite group $\Ga/\Ga'$.
Then
$$
\psi_{\Ga,\CD}(x)=\frac{2|\CD|}{|\Ga/\Ga'|}\[x+ \sum_{\chi\ne 1}\tr\chi(\CC)\sum_{j=1}^{n(\chi)}\frac{x^{s_{\chi,j}}}{s_{\chi,j}}\]+O\(x^{\frac34}(\log x)^{\frac12}\).
$$
In particular, if $\Ga=\SL_2(\Z)$ and $\Ga'=\ker\(\Ga\to\SL_2(\F_p)\)$ for a prime $p\ge 3$, then
$$
\psi_{\Ga,\CD}(x)=\frac{2|\CD|}{|\Ga/\Ga'|}\ x +O\(x^{\frac34}(\log x)^{\frac12}\).
$$
\end{corollary}

\section{Congruence groups}

In this section we fix a prime $p$ and consider the case $\Ga=\SL_2(\Z)$ and $\Ga'=\ker\(\SL_2(\Z)\to\SL_2(\F_p)\)$.
As the latter map is surjective, we get $\Ga/\Ga'=\SL_2(\F_p)$ and $\La/\La'=\PSL(\F_p)$. 
The following lemma is an exercise in linear algebra.
For the convenience of the reader, we include a proof.

\begin{lemma}\label{lem4.1}
Let $p\ge 3$.
In the following table, we give a list of the conjugacy classes in $\SL_2(\F_p)$.
We write $A$ for a representative of the class and $a=\tr(A)$.
In each row, we denote the number of  classes of the respective form by $d$.
In the table, $\al$ is an arbitrary element of $\F_p^\times$.
$$
\begin{array}{|c|c|c|c|c|}
\CD=[A]&d&|\cent(A)|&\[\frac{a^2-4}p\]\\
\hline
\smat\mu\ \ {1/\mu}, \ 1\ne \mu\in\F_p^\times&\frac{p-3}2&p-1&1
\\ \hline
\smat a 1{-1} \ , \ a=\mu+\frac1\mu,\ \mu\notin\F_{p}\ 
&\frac{p-1}2&p+1&-1
\\ \hline
\smat{\pm 1}\ \ {\pm 1}&2&p(p^2-1)&0
\\ \hline
\smat {\pm 1}\al\ {\pm 1}&4&2p&0
\\ \hline
\end{array}
$$
If $p=2$, the group $\SL_2(\F_p)$ has 6 elements. There are 3 conjugacy classes, two with $a=0$, making up 4 elements and one with $a=1$ consisting of two elements.
\end{lemma}

\begin{proof}
If $A$ has only one eigenvalue $\mu$ in the algebraic closure of $\F_p$, then $\mu^2=1$ and so $\mu=\pm 1$ and we get the last two lines of the table.
In particular, for the last line recall that conjugation  changes $\al$ to $\al\be^2$ for some $\be\in\F_p^\times$ and the subgroup of squares has index two in $\F_p^\times$.

If $A$ has two different eigenvalues $\mu,1/\mu$, then $\mu+1/\mu=a$ results in a quadratic equation whose solutions lie in $F_p$ if and only if $a^2-4$ is a square in $\F_p^\times$.
Line one deals with the case $\mu\in\F_p^\times$, which can take any value in $\F_p^\times$ except for $\pm 1$ and we can exchange $\mu$ and $1/\mu$,  leading to the value of $d$. The centraliser is the torus of diagonal matrices.

Finally, if the eigenvalues are not in $\F_p$, they lie in the unique quadratic extension $\F_{p^2}$ and the trace $a$ determines the eigenvalues up to $\mu\mapsto 1/\mu$. 
Let $A$ be the matrix as in line two.
The centraliser $\CA$ of $A$ in the matrix algebra $\M_2(\F_p)$ has dimension 2 and is a skew field, since $a$ is a non-square. Hence the unit group $\CA^\times$ has $(p^2-1)$ elements. The determinant $\det:\CA^\times\to\F_p^\times$ is given by a non-degenerate quadratic form, hence is surjective \cite{Cass}, and so its kernel $\CA^1$ has $(p^2-1)/(p-1)=p+1$ elements.
Finally, to determine $d$ in this case, note that, as the determinant $\CA^\times\to\F_p^\times$ is surjective, any $B\in\SL_2(\F_p)$, which is conjugate to $A$ in $\GL_2(\F_{p^2}$, also is conjugate to $A$ in $\SL_2(\F_p)$.
Therefore the number $a$ determines the conjugacy class, as $\mu\ne1/\mu$.
The trace  $a=\mu+1/\mu\in\F_p$ is fixed under the  Frobenius and as the trace determines the eigenvalues, the Frobenius must interchange $\mu$ and $1/\mu$, hence $\mu^{p+1}=1$. In the cyclic group $\F_{p^2}^\times$, there are $p+1$ possibilities for $\mu$ and again we can switch $\mu$ and $1/\mu$, so $d=(p-1)/2$.
\end{proof}

Note that
$$
\psi_a^\pm=\frac12\sum_{\substack{\CD\subset\Ga/\Ga'\\ \tr(\CD)=\pm a}}\psi_{\Ga,\CD}.
$$
With this, Theorem \ref{thmMain}  follows from 
Corollary \ref{cor3.3}
and Lemma \ref{lem4.1}.

{\small Mathematisches Institut\\
auf der Morgenstelle 10\\
72076 T\"ubingen\\
Germany\\
\tt deitmar@uni-tuebingen.de}

\today

\end{document}